\documentclass[11pt,a4paper]{amsart}
\usepackage{amssymb,amsmath,layout}

\theoremstyle{plain}
\newtheorem{thrm}{Theorem}[section]
\newtheorem{lemma}[thrm]{Lemma}
\newtheorem{prop}[thrm]{Proposition}
\newtheorem{cor}[thrm]{Corollary}
\newtheorem{rmrk}[thrm]{Remark}
\newtheorem{dfn}[thrm]{Definition}

\begin{document}
% begin top matter
% ***************** macroes needed for this paper ************************

\newcommand{\SL}{\mathcal L^{1,p}( D)}
\newcommand{\CO}{C^\infty_0( \Omega)}
\newcommand{\Rn}{\mathbb R^n}
\newcommand{\Rm}{\mathbb R^m}
\newcommand{\R}{\mathbb R}
\newcommand{\Om}{\Omega}
\newcommand{\Hn}{\mathbb H^n}
\newcommand{\aB}{\alpha B}
\newcommand{\eps}{\ve}
\newcommand{\BVX}{BV_X(\Omega)}
\newcommand{\p}{\partial}
\newcommand{\IO}{\int_\Omega}
\newcommand{\bG}{\boldsymbol{G}}
\newcommand{\bg}{\mathfrak g}
\newcommand{\bz}{\mathfrak z}
\newcommand{\bv}{\mathfrak v}
\newcommand{\Bux}{\mbox{Box}}
\newcommand{\e}{\ve}
\newcommand{\X}{\mathcal X}
\newcommand{\Y}{\mathcal Y}
\newcommand{\W}{\mathcal W}
\newcommand{\vt}{\vartheta}
\newcommand{\vf}{\varphi}

\numberwithin{equation}{section}

\newcommand{\RN} {\mathbb{R}^N}
\newcommand{\Sob}{S^{1,p}(\Omega)}
\newcommand{\Dxk}{\frac{\partial}{\partial x_k}}
\newcommand{\Co}{C^\infty_0(\Omega)}
\newcommand{\Je}{J_\ve}
\newcommand{\eh}{\ve h}
\newcommand{\Dxi}{\frac{\partial}{\partial x_{i}}}
\newcommand{\Dyi}{\frac{\partial}{\partial y_{i}}}
\newcommand{\Dt}{\frac{\partial}{\partial t}}
\newcommand{\aBa}{(\alpha+1)B}
\newcommand{\GF}{\psi^{1+\frac{1}{2\alpha}}}
\newcommand{\GS}{\psi^{\frac12}}
\newcommand{\HFF}{\frac{\psi}{\rho}}
\newcommand{\HSS}{\frac{\psi}{\rho}}
\newcommand{\HFS}{\rho\psi^{\frac12-\frac{1}{2\alpha}}}
\newcommand{\HSF}{\frac{\psi^{\frac32+\frac{1}{2\alpha}}}{\rho}}
\newcommand{\AF}{\rho}
\newcommand{\AR}{\rho{\psi}^{\frac{1}{2}+\frac{1}{2\alpha}}}
\newcommand{\PF}{\alpha\frac{\psi}{|x|}}
\newcommand{\PS}{\alpha\frac{\psi}{\rho}}
\newcommand{\ds}{\displaystyle}
\newcommand{\Zt}{{\mathcal Z}^{t}}
\newcommand{\XPSI}{2\alpha\psi \begin{pmatrix} \frac{x}{|x|^2}\\ 0 \end{pmatrix} - 2\alpha\frac{{\psi}^2}{\rho^2}\begin{pmatrix} x \\ (\alpha +1)|x|^{-\alpha}y \end{pmatrix}}
\newcommand{\Z}{ \begin{pmatrix} x \\ (\alpha + 1)|x|^{-\alpha}y \end{pmatrix} }
\newcommand{\ZZ}{ \begin{pmatrix} xx^{t} & (\alpha + 1)|x|^{-\alpha}x y^{t}\\
     (\alpha + 1)|x|^{-\alpha}x^{t} y &   (\alpha + 1)^2  |x|^{-2\alpha}yy^{t}\end{pmatrix}}
\newcommand{\norm}[1]{\lVert#1 \rVert}
\newcommand{\ve}{\varepsilon}
\newcommand{\Mp}{\mathcal P^+_{\la,\Lambda}}
\newcommand{\Mm}{\mathcal P^-_{\la,\Lambda}}
\newcommand{\la}{\lambda}
\newcommand{\La}{\Lambda}
\newcommand{\Sn}{\mathcal S_n}
\newcommand{\Lp}{\mathcal L^+}
\newcommand{\Lm}{\mathcal L^-}

\title[Boundary Harnack etc]{Boundary behavior of nonnegative solutions of fully nonlinear parabolic equations}

\author{Agnid Banerjee}
\address{Department of Mathematics\\Purdue University \\
West Lafayette, IN 47907} \email[Agnid Banerjee]{banerja@math.purdue.edu}
\thanks{First author supported in part by the second author's NSF Grant DMS-1001317, and in part by the second author's Purdue Research Foundation Grant \emph{``Gradient bounds, monotonicity of the energy for some nonlinear singular diffusion equations, and
unique continuation''}, 2012}

\author{Nicola Garofalo}
\address{Dipartimento di Ingegneria Civile, Edile e Ambientale (DICEA) \\ Universit\`a di Padova\\ 35131 Padova, ITALY}
\address{Department of Mathematics\\
Purdue University \\
West Lafayette IN 47907-1968}

\email[Nicola
Garofalo]{rembrandt54@gmail.com}

%
% 
% AMS information
%

\maketitle
% end top matter

\section{Introduction}\label{S:intro}

The study of the boundary behavior of nonnegative solutions of elliptic or parabolic equations occupies a central position in the theory of second order pde's. One fundamental result states that if $u$ and $v$ are two nonnegative harmonic functions in a Lipschitz domain, and $u$ and $v$ vanish continuously on a portion of the boundary, then they must converge to zero at the same rate. A quantitative, scale-invariant version of this statement is known as \emph{comparison theorem}, and it was first proved by Dahlberg \cite{D} and Wu \cite{Wu}, and subsequently generalized by Jerison and Kenig to non-tangentially accessible (NTA) domains in \cite{JK}. At the same time, in \cite{CFMS} Caffarelli, Fabes, Mortola and Salsa extended this principle to solutions of divergence form elliptic equations in Lipschitz domains, whereas in \cite{B} Bauman proved the comparison principle for solutions of nondivergence form elliptic equations in Lipschitz domains. The proof of the comparison  theorem hinges crucially on the so-called \emph{boundary Harnack principle}, or \emph{Carleson estimate}, which states that for a nonnegative harmonic function which vanishes continuously on a portion of the boundary, then its values in an interior region and up to that boundary portion are controlled by the value of the function itself at a suitably scaled interior point. 

The first version of the boundary Harnack principle for parabolic equations appeared in Kemper's paper
\cite{Ke} on the non-tangential convergence of solutions of the heat equation in Lip$(1,\frac 12)$ domains in $\R^{n+1}$. Such result was subsequently extended by Salsa in \cite{S} to divergence form parabolic equations in Lipschitz cylinders, and by the second named author in \cite{G} to non-divergence form parabolic equations in Lipschitz cylinders. In the same paper a parabolic version of the comparison theorem was first established for $C^2$ cylinders. In \cite{G} it was also proved for the first-time that nonnegative solutions vanishing on the lateral boundary of non-divergence form parabolic equations with time-independent coefficients satisfy a basic elliptic type Harnack inequality, the so-called \emph{backward Harnack inequality}. In \cite{FGS1} the parabolic comparison theorem was obtained for the heat equation in Lip$(1,\frac 12)$ domains in $\R^{n+1}$, whereas in \cite{FGS2} this result, the ensuing doubling property of the caloric measure, and the backward Harnack inequality for equations with time-independent coefficients were established for divergence form parabolic equations in Lipschitz cylinders.  In the paper \cite{FS} Fabes and Safonov first succeeded in proving the backward Harnack inequality for divergence form equations with time-dependent coefficients. Subsequently, this result was generalized to nondivergence form parabolic equations by Fabes, Safonov and Yuan in \cite{FSY}. A unified approach to the backward Harnack inequality and the comparison theorem, which covers at the same time variational and nonvariational parabolic equations, was developed by Safonov and Yuan in the paper \cite{SY}. 

In their recent paper \cite{ASS} Armstrong, Sirakov and Smart have established the comparison theorem for solutions of fully nonlinear elliptic equations, in fact more generally for solutions to differential inequalities for the Pucci extremal operators in $C^{2}$ domains. The authors adapt the proof of Theorem 1.3 in \cite{BCN}, see Proposition 2.1 in \cite{ASS}. The comparison theorem has been one of the crucial tools in \cite{ASS} in studying the behavior of  viscosity solutions of fully nonlinear elliptic equations near singular  boundary points. The same result for smooth domains has also been employed in the recent interesting paper \cite{AS} to establish the strong unique continuation property for solutions to fully nonlinear elliptic equations under some structural assumptions. 

These recent results in fully nonlinear elliptic equations, and the desire of investigating their parabolic counterpart following the historical development for linear equations outlined above, have provided a first natural motivation for us in studying the boundary behavior of nonnegative solutions of fully nonlinear parabolic equations.  In this paper we consider fully nonlinear parabolic equations in $\R^{n+1}$ of the type  
\begin{equation}\label{main0}
F(D^2 u, Du, x, t) = u_{t},
\end{equation}
where $F$ satisfies the structural conditions in \eqref{up} below, along with $F(0,0,x,t) = 0$. Equations such as \eqref{main0} are the fully nonlinear counterpart of the non-divergence form linear equation
\begin{equation}\label{mainlin}
\sum_{i,j=1}^n a_{ij}(x,t)D_{ij} u  + \sum_{i=1}^n b_i u_i = u_{t}.
\end{equation}
Viscosity solutions to \eqref{main0} are a subclass of the (viscosity) solutions to the differential inequalities in \eqref{maine} below. Such inequalities  involve the parabolic extremal Pucci operators and the class of their viscosity solutions will be denoted by the notation $\mathcal S(\la,\Lambda,a)$,  see Definition \ref{D:pucci} and more in general Section \ref{S:prelim} below for a detailed discussion of this aspect. 

The primary purpose of the present paper is studying the boundary behavior of nonnegative functions in the class $\mathcal S(\la,\Lambda,a)$. Since, as we have just observed, viscosity solutions to \eqref{main0} are in this class, it follows that our results provide corresponding statements for nonnegative viscosity solutions to \eqref{main0} as a  special case. In this respect, we would like to emphasize that the results in the present paper encompass a much larger class of equations, than just fully nonlinear uniformly parabolic equations such as \eqref{main0}.  For instance, in Section \ref{S:app}  we show an application of our results to the following nonlinear singular  diffusion equation
\begin{equation}\label{e:p}
|Du|^{2-p} \text{div}(|Du|^{p-2}Du)=u_t,\ \ \ \ \ \ \ \ 1<p<\infty,
\end{equation}
which has recently been studied in \cite{MPR}, \cite{Do} and \cite{BG}. This is an evolution associated with the $p$-Laplacian that generalizes the motion by mean curvature, which corresponds to the case $p=1$. Viscosity solutions of the equation \eqref{e:p} satisfy (in the viscosity sense) the double differential inequality \eqref{maine} below, and therefore belong to the class $\mathcal S(\la,\Lambda,a)$, with 
\[
\la = \min\{1,p-1\},\ \ \Lambda = \max\{1,p-1\},\ \ \ a = 0.
\]
However, the equation \eqref{e:p} does not lend itself to be cast in the form \eqref{main0} above. 
Since the results in this paper pertain functions in the general class $\mathcal S(\la,\Lambda,a)$, it follows that they encompass the degenerate evolution \eqref{e:p} as well. 

Apart from their relevance in the historical development of the subject, and their applications to degenerate diffusion equations such as \eqref{e:p}, we believe that our results constitute a first step in: 
\begin{itemize}
\item[(i)] studying the parabolic counterpart of the results obtained in \cite{ASS}; 
\item[(ii)] establishing space-like unique continuation for solutions to fully nonlinear parabolic  equations. We mention in this connection the recent paper \cite{Ba} in which the first named author has obtained a partial generalization to parabolic fully nonlinear equations of the unique continuation result in \cite{AS};
\item[(iii)] providing some of the basic tools necessary to the study of free boundary problems for fully nonlinear parabolic equations.
\end{itemize} 

This paper is organized as follows. In Section \ref{S:prelim} we introduce the relevant notions and gather some known results for fully nonlinear parabolic equations. In Section \ref{S:ce},  by employing the barrier in \cite{M} and  using ideas from \cite{G}, we prove the Carleson estimate, or  boundary Harnack principle, for nonnegative functions in the class $\mathcal S(\la,\Lambda,a)$ in Lipschitz cylinders, see Theorem \ref{Carleson} below. As an application, we establish an elliptic type Harnack inequality for nonnegative solutions vanishing on the lateral boundary, see Theorem \ref{boundary max} below. In Section \ref{S:C11}  we establish our main result: the local comparison theorem for nonnegative functions in the class $\mathcal S(\la,\Lambda,a)$ in $C^{1,1}$ cylinders, see Theorem \ref{Boundary comparison theorem} below. We show that the barriers in \cite{G} can be adapted to functions in the class $\mathcal S(\la,\Lambda,a)$. In this context, we would like to mention that although the equation \eqref{main0} above is not invariant under translations in space and time, however the class $\mathcal S(\la,\Lambda,a)$  does possess such invariance. It is this simple yet remarkable fact that allows us to establish the backward Harnack inequality for nonnegative functions in $\mathcal S(\la,\Lambda,a)$ vanishing on the lateral boundary, see Theorem \ref{backward harnack} below. In closing, we mention that extending the comparison Theorem \ref{Boundary comparison theorem} to cylinders with a Lipschitz, or even non-tangentially accessible base, remains a challenging open problem. To the best of our knowledge, the comparison theorem is not known even in the elliptic case. Such results would have important applications to free boundary problems for fully nonlinear equations.  We intend to come back to these questions in a future study. 

\vskip 0.2in

\noindent \textbf{Acknowledgment:} We would like to thank Adrzej \'Swiech for bringing reference \cite{CKS} to our attention and for a helpful overview of the parabolic viscosity theory.

%>>>>>>>>>>>>>>>>>>>>>>>>>>>>>>>>>>>>

\section{Preliminary results}\label{S:prelim}

We consider fully nonlinear parabolic equations in $\R^{n+1}$ of the type  
\begin{equation}\label{main}
F(D^2 u, Du, x, t) = u_{t},
\end{equation}
with an $F$ which satisfies $F(0,0,x,t) = 0$. Hereafter, the notation $\Sn$ will indicate the space of (real) $n\times n$ symmetric matrices $M$. Given a bounded open set $\Om\subset \Rn$, and $T>0$, by saying that \eqref{main} is uniformly parabolic in the cylinder $C_T = \Om\times (0,T)$ we mean that  there exist $0<\lambda\le \Lambda$, and $a \ge 0$, such that for every $N\in \Sn$, $N\ge 0$, every $p, q\in \Rn$, and $(x,t)\in C_T$, one has  
\begin{equation}\label{up}
n\lambda ||N|| - a |p - q|\leq F(M+N,p,x, t) - F(M,q, x, t) \leq  \Lambda ||N||  + a |p -q|,
\end{equation}
where we have let $||M|| = \underset{|x|=1}{\sup} |Mx|$. 

Whenever convenient, we will use the notation $z = (x,t), z_0 = (x_0,t_0), \zeta = (\xi,\tau)$ to indicate points in $\R^{n+1}$. We now introduce the relevant notion of subsolutions and supersolutions, see p. 34 in \cite{W} and p. 2004 in \cite{CKS}. Let $\Om$ be an arbitrary open set in $\R^{n+1}$. By a parabolic  neighborhood of a point $z_{0}= (x_0, t_0)$, we  mean the interesection of an Euclidean neighborhood $U$ of $z_{0}$ with $\Rn \times ( -\infty, t_{0}]$. In the following definition, a local extremum is to be understood with respect to parabolic neighborhoods.

\begin{dfn}\label{D:vss}
Let $D\subset \R^{n+1}$ be a bounded open set. A function $u\in C(D)$ is said to be  a \emph{viscosity subsolution} to \eqref{main} in $D$ if, for given  $\vf \in C^{2} (D)$, at a local maximum $z_0$ of $u - \vf$, we have 
\begin{equation*}\label{vsubs}
\vf_t(z_0)  \leq F(D^2 \vf, D\vf,u,z_0).
\end{equation*}
A function $u\in C(D)$ is said to be  a \emph{viscosity supersolution} to \eqref{main} if, for given  $\vf \in C^{2}(D)$, at a local minimum  $z_0$ of $u - \vf$, we have 
\begin{equation}
\vf_t(z_0)  \geq F(D^2\vf, D\vf,u,z_0)
\end{equation}
If $u$ is both a viscosity subsolution and supersolution, then we say that it is a viscosity solution to \eqref{main} in $D$.
\end{dfn}

Given $0<\la\le \Lambda$, we will denote by 
\[
[[\la,\Lambda]] = \{A\in \mathcal S_n\mid \la \mathbb I_n \le A \le \Lambda \mathbb I_n\}.
\] 
We let $\Mp$ and $\Mm$ denote the maximal and minimal \emph{Pucci extremal operators} corresponding to $\lambda,  \Lambda$, i.e., for every $M\in \mathcal S_n$ we have
\[
\Mp(M) = \Mp(M,\lambda,\Lambda) = \La \sum_{e_i>0} e_i + \la \sum_{e_i<0} e_i,
\]
\[
\Mm(M) = \Mm(M,\lambda,\Lambda) = \la \sum_{e_i>0} e_i + \La \sum_{e_i<0} e_i,
\]
where $e_i = e_i(M)$, $i=1,...,n$,  indicate the eigenvalues of $M$. As is well known, $\Mp$ and $\Mm$ are uniformly elliptic fully nonlinear operators, and moreover
\begin{equation}\label{pucci}
\Mp(M) = \underset{A\in [[\la,\Lambda]]}{\sup}\ \left[\text{trace}(AM)\right],\ \ \ \ \Mm(M) = \underset{A\in [[\la,\Lambda]]}{\inf}\ \left[\text{trace}(AM)\right].
\end{equation}

What is important for us, see Proposition 3.10 in \cite{W}, is that a viscosity solution $u$ to \eqref{main} satisfies the following differential inequalities in the viscosity sense
\begin{equation}\label{maine}
\Mp(D^2 u) + a|Du|  - u_t\geq 0 \geq \Mm(D^2 u) -  a|Du|  -  u_t,
\end{equation}
where $a\ge 0$ is the number appearing in \eqref{up}.
Henceforth, we write
\begin{equation}\label{lpm}
\Lp u =  \Mp(D^2 u) + a|Du|,\ \ \ \Lm u  =  \Mp(D^2 u) - a|Du|.
\end{equation}

\begin{dfn}\label{D:pucci}
Given a bounded open set $D\subset \R^{n+1}$, the symbol $\mathcal S(\la,\Lambda,a)$ will indicate the class of functions $u\in C(D)$ which are viscosity solutions of \eqref{maine} in the sense that $u$ is at the same time a \emph{viscosity subsolution} of the fully nonlinear equation 
\begin{equation}\label{Lp}
\Lp u - u_t = 0,
\end{equation}
and a \emph{viscosity supersolution} of
\begin{equation}\label{Lm}
\Lm u - u_t = 0.
\end{equation}
\end{dfn}

\begin{rmrk}\label{R:pucci}
We note that, unlike that of solutions to \eqref{main} above, the class $\mathcal S(\la,\Lambda,a)$ in $\R^{n+1}$  is invariant under orthogonal transformations in the space variable and also, importantly, it is invariant under the time translations $t\to u(\cdot,t+h)$, $h\in \R$ fixed. This simple, yet remarkable, fact will play an important role in Section \ref{S:C11} below.
\end{rmrk}

For further details on the viscosity theory for functions in the class $\mathcal S(\la,\Lambda,a)$ we refer the reader to section 3 in \cite{W} or section 1 in \cite{CKS}.
We have observed above that if $u$ is a viscosity solution of \eqref{main}, then $u\in \mathcal S(\la,\Lambda,a)$ (the opposite implication is not true, see the discussion in the introduction, and Section \ref{SS:p} below). Since all the results in the present paper pertain to functions $u\in \mathcal S(\la,\Lambda,a)$, we infer that our results encompass viscosity solutions of the equation \eqref{main} as well.
 
For equations of the form \eqref{main}, we have the following comparison principle (see for instance the paper \cite{GGIS}, which covers  more general degenerate equations and even unbounded domains). 

\begin{rmrk}\label{R:comp}
We note that, since the fully nonlinear equations \eqref{Lp} and \eqref{Lm} are special instances of \eqref{main}, Theorem \ref{comparison} and Corollary \ref{comp} below are also valid for the operators \eqref{Lp}, \eqref{Lm}. 
\end{rmrk}

\begin{thrm}\label{comparison}
Let $u$ be a  viscosity subsolution and $v$ be a viscosity supersolution in $C_T$ of the fully nonlinear equation \eqref{main}. If $ u \leq v$ on $\p_p C_T$, then $u\leq v$ in $C_T$.
\end{thrm}

The above comparison theorem has the following corollary, which will be subsequently used in this paper.
\begin{cor}\label{comp}
Let $D\subset \R^{n+1}$ be a bounded open set, let $u$ be a viscosity subsolution and $v$ be a viscosity supersolution in $D$ of \eqref{main}. If $u \leq v $ on $\partial D$, then $ u \leq v $ in $\overline{D}$.
\end{cor}

\begin{proof}
When $D$ is a parabolic cylinder $C_T$, then the conclusion follows from Theorem \ref{comparison} above. It easily follows that Corollary \ref{comp} continues to hold for a finite union of such cylinders. In order to see this,  consider
$K= \bigcup_{i=1}^M Q_i \times (t_{1i}, t_{2i})$ and assume that $ u \leq v $ on $\partial K$. Take points   $ s_1 < s_2 < s_3 <...<s_N$ in the increasing order such that $s_{k} =t_{ij}$ for some  $i = 1, 2,  j = 1, .., M$.  Consider the set $K \cap [ s_1 < t < s_2]$ which is  a disjoint  union of sets of the form $\Om_k \times (s_1, s_2)$ where each $\Om_k $ is a   connected domain (union of subcollection of cubes $Q_i$'s). Moreover, $u \leq v $ on $\partial_{p}(\Om_k  \times (s_1, s_2))$ for each $k$. By Theorem \eqref{comparison}, we thus have $u \leq v $ in $\overline{\Om_{k} \times (s_1,s_2)}$. We conclude that  $ u \leq v $ in  $K \cap [t \leq s_2]$. Now, we proceed inductively to obtain the same conclusion in $K \cap (s_i \leq t \leq s_{i+1}]$, and the claim follows.
For a general domain $D$, given $ \ve > 0$, consider the set
\[
K = \{(x,t) \in \Omega \mid u(x,t) \geq v(x,t) + \ve\}.
\]
Suppose $K\not= \varnothing$. In view of the assumption $u\le v$ on $\p D$, the set $K$ is compactly contained in $D$. Therefore, there is an open set $\tilde D \subset D$ such that $K \subset \tilde D$ and $\tilde D$ is a union of finitely many cylinders. On $\partial \tilde D$, we have 
$u(x,t) \leq v(x,t) + \ve$  (note that $ v + \ve $ is also a supersolution of \eqref{Lp}, or \eqref{Lm}).
From what was said above we conclude that $u(x,t) \leq v(x,t) + \ve$ in $\tilde D$, a contradiction to the fact that $K \subset \tilde D$. Therefore, $K = \varnothing$, and we thus have  $u < v + \ve$ for all $\ve > 0$. The desired conclusion follows by letting $\ve \to 0$.

\end{proof}

\begin{rmrk}\label{R:universal}
Hereafter in this paper when we say that a constant is \emph{universal} we mean that it depends only on the dimension $n$, and the parameters $\lambda, \Lambda, a$ in \eqref{maine}.
\end{rmrk}

Now fix $\vt\in (0,\pi)$, and $r>0$. We will denote by $\Gamma_{\vt,r}$ the \emph{truncated cone} of radius $r$, vertex at $0$, aperture $\vt$ and axis along the $x_n$ coordinate axis, defined by 
\[
\Gamma_{\vt,r} = B_r(0) \cap \left(\left\{x\in \Rn\mid \arccos\left(\frac{x_n}{|x|}\right) \le \vt\right\} \cup \{0\}\right).
\]
We denote by $\Gamma_{\vt,\infty} = \left\{x\in \Rn\mid \arccos\left(\frac{x_n}{|x|}\right) \le \vt\right\} \cup \{0\}$ the infinite cone corresponding to the case $r = \infty$. By translating and rotating $\Gamma_{\vt,r}$ we obtain a truncated cone $\Gamma_{\vt,r}(x_0)$ with vertex at any other point $x_0\in \Rn$. We need the following barriers on $\Gamma_{\vt,r}$ for the operator $\Lp$. The existence of such barriers follows directly from a result of Miller in \cite{M} for non-divergence form elliptic operators with bounded measurable coefficients.

\begin{prop}\label{P:bd}
Given  $\vt\in (0,\pi)$, there exist a universal $R_0 > 0$, depending also on $\vt$,  and a uniform barrier $\vf$ for $\Gamma_{\vt, R_0}$ and  the operator $\Lp$. This means that $\vf\in C^2$ in the interior of $\Gamma_{\vt, R_0}$ and that we have for some universal $C, \alpha, \mu_1, K>0$, depending also on $\vt$:
\begin{itemize}
\item[1)] $0 \leq \vf(x) \leq C |x|^{\alpha}$;
\item[2)] $\vf(0)= 0$;
\item[3)] $\vf( x) \geq \mu_1$, $|x|= R_0$;
\item[4)] $\Lp \vf(x)\leq -  K|x|^{\alpha - 2}$.
\end{itemize}
\end{prop}

\begin{proof}
 The barrier in Theorem 2 in \cite{M} is of the form $\vf(x)= |x|^{\alpha} h\left(\arccos\big(\frac{x_n}{|x|}\big)\right)$. It is proved in \cite{M} that, for all  matrices $[a_{ij}]\in [[\lambda, \Lambda]]$, one has in $\Gamma_{\vt,R_0}\setminus\{0\}$ 
 \[
 a_{ij}D_{ij}\vf(x) \leq - A |x|^{\alpha - 2}, \ \ \ \ \ \ |D\vf(x)| \leq B |x|^{\alpha - 1},
 \]
where $A, B>0$ are universal constants depending also on $\vt$. Since by \eqref{pucci} we have
\[
\Mp (D^2 \vf)(x) = \underset{\lambda \mathbb I_n \leq  [a_{ij}] \leq \Lambda \mathbb I_n }{\sup}\ a_{ij} D_{ij} \vf(x),
\]
we thus obtain from \eqref{lpm} 
\[
\Lp \vf =  \Mp(D^2 \vf) + a|D\vf| \le - A |x|^{\alpha - 2} + a B  |x|^{\alpha - 1}  = - A |x|^{\alpha -2} [1 - A^{-1} a B |x|].
\]
Since $|x|\le R_0$ for $x\in \Gamma_{\vt,R_0}$, it suffices to choose for instance $R_0\le \frac{A}{2aB}$ to obtain the desired conclusion with $K = A/2$.

\end{proof}

Now, if $\Gamma_{\vt,r}$ indicates a truncated cone of radius $r$, aperture $\vt$, and vertex at $0$, we denote by
\[
C_{\vt,r}= \Gamma_{\vt,r} \times (-r^2, r^2)
\]
the \emph{parabolic wedge} with vertex at $(0,0)\in \R^{n+1}$, radius $r$, aperture $\vt$ and height $2r^2$. Using Proposition \ref{P:bd}, and adapting Theorem $1.12$ in \cite{G} for linear parabolic operators in non-divergence form, we can construct a barrier for  $C_{\vt,r}$ and the fully nonlinear parabolic operator $\Lp - \p_t$.  

\begin{thrm}\label{barrier}
There is a universal $R_0 <1$, depending also on $\theta$, such that for $0<r \le R_0$, there exists a uniform barrier for $\Lp - \partial_t $ on $C_{\vt,r}$. This means that we can find a function $\psi \in C^{2}$ in the interior of $C_{\vt,r}$ for which we have for some universal $C, \alpha, \mu>0$, depending also on $\vt$:
\begin{itemize}
\item[1)] $0 \leq  \psi(x,t) \leq C\left(\frac{|x| + |t|^{1/2}}{r}\right)^{\alpha}$;
\item[2)] $\psi (0,0) = 0$;
\item[3)] $\psi(x, t) \geq \mu $ for each $(x, t) \in \partial C_{\vt,r}$ such that either $|x|= r$, or $t= - r^2$;
\item[4)] $\Lp (\psi) - \psi_t < 0$. 
\end{itemize}
\end{thrm}

\begin{proof}
We begin by observing that the theorem holds true when $r= R_0$ where $R_0$ is as in Proposition \eqref{P:bd} and then analyze the general case. Without loss of generality, we can assume $R_0 < 1 $. We  define 
\[
\psi(x,t) = \vf(x) +  \frac{K}{2\kappa} |t|^{\kappa},
\]
where $\kappa\in \mathbb N$ with $\kappa>2$.
From Proposition \ref{P:bd} above, by arguing as in the proof of Theorem 1.12 in \cite{G}, it is easily seen that we can reach the desired conclusion for $r=R_0$.  For the general case, we define $\psi_1(x,t) = \psi (\frac{R_0}{r}x,\frac{R_0^2 }{r^2}t)$ for $(x,t) \in C_{\vt,r}$. Let us notice that $(x,t) \in C_{\vt,r}$ if and only if $(\frac{R_0}{r}x, \frac{R_0^2 }{r^2}t) \in C_{\vt,R_0}$. Then, we have 
\[
\left[\Lp (\psi_1) -(\psi_1)_t\right](x,t) = \frac{R_0^2}{r^2} \left[\Mp(D^2 \psi) + \frac{r}{R_0} a |D\psi|  - \psi_t\right]\big(\frac{R_0}{r}x,\frac{R_0^2 }{r^2}t\big).
\]
Now, when $0<r\le R_0$ we have $\frac{r}{R_0}\leq 1$, and since $(\frac{R_0}{r}x, \frac{R_0^2 }{r^2}t) \in C_{\vt,R_0}$, by the properties 4) of $\psi$ in $C_{\vt,R_0}$ we conclude
\[
\left[\Lp (\psi_1) -(\psi_1)_t\right](x,t) \le \frac{R_0^2}{r^2} \left[\Mp(D^2 \psi) + a |D\psi|  - \psi_t\right]\big(\frac{R_0}{r}x,\frac{R_0^2 }{r^2}t\big) <0.
\]
The properties 1)-3) are also easily verified for the function $\psi_1$ in the wedge $C_{\vt,r}$. Therefore, the conclusion follows.

\end{proof}

We close this section by quoting the important Krylov-Safonov type Harnack inequality for solutions of the differential inequalities \eqref{maine}. For this result we refer the reader to Theorem 4.18 in \cite{W}, when no lower order terms are present, and to Section 4.6 in the same paper for a discussion on how to handle lower order terms. One should also see Lemma 5.2 in \cite{CKS} for the explicit statement. For $\eta >1 $ and $r \in [0, 1]$, define 
\[
Q(\eta, r) = \{(x,t)\in \R^{n+1}\mid \max_{1 \leq i \leq n} |x_{i}| \leq r,\ 0 < t \leq \eta r^2\}.
\]

\begin{thrm}\label{Harnack}
Let $u\ge 0$ be a function in the class $\mathcal S(\la,\Lambda,a)$ in $Q(\eta, r)$ and $0 < \sigma < 1$. Then, there exists a universal constant $C= C(\eta,\sigma)>0$ such that
\begin{equation}
\max_{ |x| \leq \sigma r } u(x, r^2) \leq C \min_{|x| \leq \sigma r} u(x,\eta r^2).
\end{equation}
\end{thrm}

\section{Boundary Harnack Principle for Lipschitz cylinders}\label{S:ce}

 In this section, we work on a cylinder $C_T = \Om \times (0,T)$, where $\Om \subset \Rn$ is a Lipschitz domain. 
We recall that a bounded domain $\Om$ is a Lipschitz domain if for each $Q \in \partial \Om$ there exist a ball $B_{r_0}(Q)$ centered at $Q$ and a coordinate system for $R^{n}$ such that in these coordinates
\[
B_{r_0}(Q) \cap \Om = B_{r_0} \cap \{(x', x_n): x' \in \R^{n-1}, x_n > \phi (x')\},
\]
\[
B_{r_0}(Q) \cap \p \Om = B_{r_0} \cap \{(x', x_n): x' \in \R^{n-1}, x_n = \phi (x')\},
\]
where $\phi:\R^{n-1}\to \R$ is a Lipschitz continuous function with $||D\phi||_{L^\infty(\R^{n-1})} \leq m$.
By compactness we can, and will, assume that the radius $r_0$ and the constant $m$ are independent of $Q\in \p \Om$. Such constants determine what will be called the Lipschitz character of $\Om$. We further indicate with $S_T= \partial \Om \times (0, T)$ the lateral boundary of $C_T$, while $\partial_{p} C_T= S_T \cup (\Om \times \{0\})$.
For $(Q,s)\in S_T$ and sufficiently small $r \leq \frac{r_0}{10}$, we define 
\[
\Psi_r(Q,s)= \{(x,t) \in \overline{C}_T\mid |x-Q| <r, |t-s|<r^2\},
\]
\[
\Delta_r(Q,s)=  \Psi_r(Q,s) \cap \partial_{p}C_T.
\]
If $Q\in \partial \Om$ is given by $(x'_0, \phi(x'_0))$ in the above local coordinates, for $r>0$ small enough we set 
\[
\overline{A}_r(Q,s)= (x'_0, \phi(x'_0)+ r, s+ 2r^2),\ \ \ \ \underline{A}_r(Q,s)= (x'_0, \phi(x'_0)+  r, s- 2r^2).
\]

A Lipschitz domain $\Om$ has a uniform exterior (and interior) cone $\Gamma_{\vt, \infty}(Q)$  at every point $Q\in \p \Om$. The uniform aperture $\vt\in (0,\pi)$ of such cones depends on the Lipschitz constant $m$ of the local defining function $\phi$ of $\p \Om$.

By adapting the proof of Lemma 2.1 in \cite{G}, we next show that a solution to the fully nonlinear equation \eqref{maine} which continuously vanishes in a surface neighborhood  of a point on $S_{T}$ is actually H\"older continuous up to the boundary.

\begin{lemma}\label{Holder}
Let $(Q,s)\in S_T$. There is a universal $R_0>0$, depending  also on the Lipschitz character of $\Om$, such that for $0<r<R_0$, and sufficiently small depending also on $s, T - s, r_0$, and any  $u\ge 0$ in the class $\mathcal S(\la,\Lambda,a)$ in $\Psi_{r}(Q,s)$ and vanishing continuously on $\Delta_{r}(Q,s)$, we have for universal $C, \alpha>0$, depending also on the Lipschitz character of $\Om$, 
\begin{equation}\label{holder}
u(x,t) \leq C \left(\frac{|x-Q| + |t- s |^{1/2}}{r}\right)^{\alpha} M_{r}(u),\ \ \ \ (x,t) \in \Psi_{r}(Q,s),
\end{equation}
where $M_{r}(u)= \underset{\Psi_{r}(Q,s)}{\sup} u$.
\end{lemma}

\begin{proof}
as in Proposition \ref{P:bd},

Without loss of generality, we may assume that $M_{r}(u) < \infty$. Let  $\Gamma_{\vt, \infty}(Q)$ be a fixed infinite cone exterior to $\Om$ and having vertex at $Q$. If $\Gamma_{r/2}(Q)= B_{r/2}(Q) \setminus \Gamma_{\vt, \infty}(Q)$, let $C_{r/2}(Q,s)= \Gamma_{r/2}(Q) \times (s-r^2, s+ r^2)$ denote the parabolic wedge at $(Q,s)$,  and let $\psi$ be a uniform barrier on $C_{r/2}(Q, s)$ constructed as in Theorem \ref{barrier} above. In the cylinder $C_{T} \cap C_{r/2}(Q, s)$, we consider  
\[
w(x,t) = \frac{M_{r}(u)}{\mu} \psi(x,t),
\]
where $\mu$ is as in Theorem \ref{barrier}. By 4) in that theorem we have $\Lp(w)  -w_{t} < 0 $ in $ C_{T} \cap C_{r/2}(Q,s)$, whereas 3) implies $w \geq u $ on $ \partial_{p} (C_{T} \cap C_{r/2}(Q,s))$. Since $u$ is a subsolution and $w$ a supersolution of $\Lp - \partial_{t}$, by the comparison principle  Theorem \ref{comparison}, see Remark \ref{R:comp}, we conclude
\begin{equation}
\ u(x,t) \leq  \frac{M_{r}(u)}{\mu} \psi(x,t),\ \ \  \ \ (x,t) \in C_{T} \cap C_{r/2}(Q,s).
\end{equation}
From 1) in Theorem \ref{barrier}, we  obtain the desired bound \eqref{holder} on $u$ in $C_{T} \cap C_{r/2}(Q,s)$. If instead $(x,t) \in  \Psi_{r}(Q,s) \setminus C_{r/2}(Q,s)$, then $|x-Q| \geq r/2$.
Hence,  in that set , we have $\left(\frac{|x-Q| + |t- s |^{1/2}}{r}\right)^{\alpha} \geq 2^{-\alpha}$. Therefore for $(x,t) \in  \Psi_{r}(Q,s) \setminus C_{r/2}(Q,s)$ we trivially have 
\[
u(x,t) \leq M_{r}(u) \leq 2^{\alpha} \left(\frac{|x-Q| + |t- s |^{1/2}}{r}\right)^{\alpha}  M_{r}(u).
\]
This completes the proof.

\end{proof}

The next statement provides a quantitative information on the growth near the boundary of a nonnegative solution. We will use this result to start the worsening process in the proof of Carleson's estimate, Theorem \ref{Carleson} below. Its proof is based on the non-tangentially accessible character of a Lipschitz domain, and on the Harnack inequality in Theorem \ref{Harnack} above. Once these tools are available, one can carry the argument similarly to the proof of Lemma 2.2 in \cite{G}. We thus omit it, and refer the reader to that source for the relevant details.

\begin{lemma}\label{growth}
Let $(Q_{0}, s_{0})  \in S_{T}$, $\alpha \in (0,1)$. There exists a universal constant $C= C(m,\alpha)>0$ such that, if
$u\ge 0$  is a function in the class $\mathcal S(\la,\Lambda,a)$ in $\Psi_{2r}(Q_{0}, s_{0})$, then the following condition holds: if $(y,s), (x,t) \in \Psi_{3r/2}(Q_{0}, s_{0})$ are such that \emph{dist}$(x, \partial D)= \alpha r$, $t- s \geq r^{2}/2 $, \emph{dist}$(y, \partial D) > \frac{r}{2^h}$, where $h\in \mathbb N \cup\{0\})$ is such that $2^{-h} \leq \alpha$, then 
\begin{equation}\label{worsen}
u(y,s) \leq C^{h}u(x,t).
\end{equation}
\end{lemma}

With Lemmas  \ref{Holder} and \ref{growth} in our hands, we can argue as in the proof of Theorem 2.3 in \cite{G} and establish the following Carleson  estimate, or boundary Harnack principle. 

\begin{thrm}[Carleson estimate for Lipschitz cylinders]\label{Carleson}
Let $(Q_{0}, s_{0}) \in S_{T}$. Then, there exists a universal constant $C=C(m)>0$ such that, if $u\ge 0$ is a function in the class $\mathcal S(\la,\Lambda,a)$ in $\Psi_{2r}(Q_{0}, s_{0})$ which vanishes continuously on $\Delta_{2r}(Q_{0}, s_{0})$, then 
\begin{equation}\label{carleson est}
u(x,t) \leq C u(\overline{A}_{r}(Q_{0}, s_{0})),\ \ \ \ \ (x,t) \in \Psi_{r/8}(Q_{0}, s_{0}),
\end{equation}
where $r$ is small enough depending on $s_0, T- s_0, r_0$.
\end{thrm}

\begin{proof}
We argue as in \cite{G} and  we first  conclude \eqref{carleson est} for $r \leq R_{0}$ where $R_{0}$ is as in Lemma \eqref{Holder}. Consequently, by an application of interior Harnack inequality in Theorem \ref{Harnack}, we conclude that validity of \eqref{carleson est} for   $r$ depending only on $r_0, T-s_0, s_0$ for a possibly larger universal constant  $C>0$, depending also on the Lipschitz character of $\Om$. 

\end{proof}

As an application of the Carleson estimate, we have the following elliptic type Harnack inequality when a solution to \eqref{maine} vanishes continuously on the lateral boundary $S_T$ of the cylinder $C_T$. 

\begin{thrm}\label{boundary max}
Let $u\ge 0$ be a function in the class $\mathcal S(\la,\Lambda,a)$ in $C_T$. For $\delta>0$ set $\Om_{\delta}= \{x \in D\mid \emph{dist}(x, \partial \Om) > \delta\}$, and $C_{\delta, T}= \Om_{\delta} \times (\delta^2, T)$. There exists a universal constant $C>0$, which also depends on $m, \delta, \emph{diam}\ \Om, T$, such that, if $u$ vanishes continuously on $S_T$, then
\begin{equation}\label{max1}
\max_{\Om_{\delta, T}} u \leq C  \min_{\Om_{\delta, T}} u
\end{equation}
\end{thrm}
\begin{proof}
We note that the assumption guarantees that $u \in C(\overline{\Om_{\delta,T}})$, and that $u$ is a subsolution to $\Lp - \partial_t$ and a supersolution to $\Lm- \partial_t$ in $\Om_{\delta, T}$ respectively. Let 
\[
k_1= \underset{\partial_{p} \Om_{\delta,T}}{\min}\ u,\ \ \ \ \ \text{and}\ \  \ \ k_2=  \underset{\partial_{p}\Om_{\delta,T}}{\max}\ u.
\]
Clearly, $k_1$ is a solution (hence a subsolution) to $\Lm-\partial_t$ and $k_2$ is a solution (hence a supersolution) to $\Lp - \partial_t$. Since $k_1 \leq u$ and $k_2 \geq u$ on $\partial_{p} \Om_{\delta,T}$, by the comparison principle, Theorem \eqref{comparison},  
we obtain 
\[
k_1 = \underset{ \Om_{\delta,T}}{\min}\ u,\ \ \ \ \   k_2=  \underset{\Om_{\delta,T}}{\max}\ u.
\]
Moreover, by the compactness of $\partial_{p}\Om_{\delta, T}$, there exist $(X_0, T_0) , (X_1, T_1) \in \partial_{p} \Om_{\delta, T}$ such that
\[
u(X_0, T_0) =\underset{\partial_{p} \Om_{\delta,T}}{\min}\ u=  \min_{\Om_{\delta, T}} u,\ \ \ \text{and}\ \ \   u(X_1, T_1)=\underset{\partial_{p} \Om_{\delta,T}}{\max}\ u = \max_{\Om_{\delta, T}} u.
\]
 The rest of the proof now follows from Theorem \ref{Carleson} and the interior Harnack inequality in Theorem \ref{Harnack}, as in the proof of Theorem 2.6 in \cite{G}. 
 
\end{proof}

\section{Comparison theorem for non-negative solutions in a $C^{1,1}$ cylinder}\label{S:C11}

In this section we establish a  comparison theorem for functions in the class $\mathcal S(\la,\Lambda, a)$ in $C_T = \Om\times (0,T)$, where $\Om\subset \Rn$ is a $C^{1,1}$ domain. Our main result, Theorem {Boundary comparison theorem}, should be compared with Proposition 2.1 in \cite{ASS} in which the authors establish its elliptic counterpart for $C^2$ domains.

We recall that a $C^{1,1}$ domain is characterized by the existence of a uniform exterior and interior tangent ball. The largest radius of such balls determines what we call the $C^{1,1}$ character of the domain $\Om$. Our proof of the comparison theorem will follow by adapting  the barriers in \cite{G}  to the fully nonlinear case, and arguing similarly relying on the natural sublinearity, in place of linearity.   The following is our main result.

\begin{thrm}[Local comparison theorem for $C^{1,1}$ cylinders]\label{Boundary comparison theorem}
Let $(Q_{0}, s_{0}) \in S_{T}$. Then, there exists a universal $r_0>0$, depending also on the $C^{1,1}$ nature of $\Om$ and on $s_0, T- s_0$, and a universal constant $C>0$, depending also on the $C^{1,1}$ character of $\Om$, such that if for $0<r \le r_0$, $u, v \ge 0$ belong to $\mathcal S(\la,\Lambda,a)$ in $\Psi_{2r}(Q_{0}, s_{0})$, and $u=v=0 $ continuously on $\Delta_{2r}(Q_{0}, s_{0})$, then for all $(x,t) \in \Psi_{r}(Q_{0}, s_{0})$, one has
\begin{equation}\label{bc}
\frac{1}{C} \frac{ u(\underline{A}_{r}(Q_{0}, s_{0}))}{v(\overline{A}_{r}(Q_{0}, s_{0}))}\leq \frac{u(x,t)}{v(x,t)}\leq C \frac{ u(\overline{A}_{r}(Q_{0}, s_{0}))}{v(\underline{A}_{r}(Q_{0}, s_{0}))}.
\end{equation}
\end{thrm}

\begin{proof}
For $Q \in \partial \Om$, let $\nu_{Q}$ be the interior unit normal to $\partial \Om$ at $Q$  and for $r > 0 $ set
\begin{equation}\label{xi12}
\xi_1(Q)= Q + \frac{r}{4} \nu_{Q}  \ \ \ \    \xi_2 (Q) = Q - \frac{r}{16} \nu_{Q}.
\end{equation}
Let us define for $\delta>0$
\[
P_{\delta}(x,\tau)= \{(y,t)\mid |x-y|^2 + |t-\tau| \leq   \delta^2\}.
\]
We now choose and fix $r_0>0$, depending on the $C^{1,1}$ character of $\Om$, so that for $0<r<r_0$, and each $(Q,s) \in \Delta_{r}(Q_{0}, s_{0})$ we have
\begin{equation}\label{lowernu}
< \nu_{Q}, \nu_{Q_{0}}>   \geq \frac{1}{4},\ \ \ \ \text{for}\ \  Q \in \partial \Om, \ \ |Q- Q_{0}| \leq r,  
\end{equation}
and
\begin{equation}
\ P_{\frac r4}(\xi_{1}(Q),s) \cap S_{T}= P_{\frac{r}{16}}(\xi_{2}(Q), s) \cap S_T= \{(Q,s)\}.
\end{equation}

From now on, $(Q,s)$ is assumed to be a arbitrary point of $\Delta_{r/8}(Q_{0}, s_{0})$. We consider the set
\[
S_{r}(Q,s)= P_{\frac r8}(Q,s) \cap P_{\frac r4}(\xi_1(Q), s).
\]
Adapting the idea in \cite{G}, we wish to find two suitable barriers $h$  and $f$ such that  
\[
\Lm (h) - h_t> 0\ \ \ \  \text{in}\ \  S_{r}(Q,s),
\]
 and 
 \[
 \Lp(f) - f_t< 0\ \ \ \ \text{in}\ \  \Psi_{r}(Q_{0}, s_{0}) \cap P_{r/8}(\xi_2 (Q),s).
 \]
 With these inequalities in place, by an application of the comparison theorem (Corollary \ref{comp} above),  we then conclude the following estimates
\begin{equation}\label{b1}
\frac{u}{u(\underline{A}_{r}(Q_0, s_0)) } \geq h\  \ \ \ \text{in}\ \ S_r(Q,s),
\end{equation} 
 and
\begin{equation}\label{b2}
\frac{u}{u(\overline{A}_{r}(Q_0, s_0)) } \leq C f \ \ \ \ \text{in}\ \   \Psi_{r} \cap P_{r/8}(Q,s).
\end{equation}
We emphasize that \eqref{b2} is the crucial step. To accomplish it we will use the Carleson estimate 
\eqref{carleson est} in Theorem \ref{Carleson} above.

We assume henceforth that $r>0$ is chosen sufficiently small (depending on $r_0$, $s_0$ 
and $T-s_0$) so that all the forthcoming arguments make sense.
Observe that if  
\[
t_1= \inf \{t\mid (x,t) \in S_{r}(Q,s)\},\ \ \ \ \ t_2= \sup \{t\mid(x,t) \in S_{r}(Q,s)\},
\]
then there exist $\alpha_1, \alpha_2  > 0$ such that
\begin{equation}\label{e1}
s_{0} - 2r^2 + \alpha_1 r^2 < t_1 < t_2 < s_0 + 2r^2 - \alpha_2 r^2.
\end{equation}
We next consider the sets
\[
\Phi^{1}_{r} (Q, s) = \partial S_{r}(Q,s) \cap \partial P_{\frac r4}(\xi_1(Q),s),\ \ \ \ \ \Phi^{2}_{r} (Q,s)= \partial S_{r}(Q,s) \cap \partial P_{\frac r8}(Q,s). 
\] 
There exists $\alpha_3 > 0 $ such that
\begin{equation}\label{e2}
\text{dist}(\Phi^{2}_r, S_T) \geq \alpha_3 r.
\end{equation}
Therefore, from \eqref{e1}, \eqref{e2}  and the Harnack inequality in Theorem \ref{Harnack} above, we obtain 
\begin{equation}\label{e3}
u(x,t) \geq C u(\underline{A}_{r}(Q_0, s_0)),\ \ \ \ \ (x,t) \in \Phi^{2}_{r}(Q,s).
\end{equation}
Now, let $R(x,t)^2 \overset{def}{=} |x-\xi_1(Q)|^2 + |t-s|$. Observe that on $\partial P_{r/4} (\xi_1(Q), s)$ we have $R^2= r^2/16$, and that
\begin{equation}\label{min}
\min_{S_r(Q,s)} \frac{|x- \xi_1 (Q)|^2}{r^2} = \frac{1}{64}
\end{equation}
Now we choose $\alpha > 0 $ large enough depending on $\lambda , \Lambda , n$.
Let 
\[
C_0 = \max_{\Phi^{2}_{r}(Q,s)} \left(\exp\left(-\frac{\alpha R^2}{r^2}\right) - \exp\left(- \frac{\alpha}{16}\right)\right)  > 0.
\]
We stress that, in view of \eqref{e2}, $C_0$ is independent of $r>0$. Pick $\gamma > 0 $ such that
\begin{equation}
\gamma < \frac{C}{2C_0},
\end{equation}
where $C>0$ is the constant in \eqref{e3} above.
We now define 
\[
h(x,t) = \gamma \left(\exp\left(- \frac{\alpha R^2}{r^2}\right) - \exp\left(- \frac{\alpha}{16}\right)\right). 
\]
A verification shows that $h=0 $ on $\Phi^{1}_r(Q,s)$. Moreover, from the choice of $\gamma$ and \eqref{e3}, we have that
\begin{equation}
h(x,t) \le \gamma C_0 < \frac{C}{2} < C \leq \frac{u(x,t)}{ u(\underline{A}_{r}(Q_0, s_0))},\ \ \ (x,t) \in \Phi^{2}_{r}(Q,s).
\end{equation}
Now, for $(x,t) \in S_r(Q,s)$ we have
\begin{equation}
Dh = - \frac{2 \alpha \gamma}{r^2}\left(x - \xi_1(Q)\right)\exp\left( - \frac{\alpha R^2}{r^2}\right),
\end{equation}
and
\begin{equation}
D_{ij} h = \left(\frac{4 \alpha^2 \gamma}{r^4} \left(x_{i} - \xi_{1}(Q)_{i}\right) \left(x_{j} - \xi_{1}(Q)_{j}\right) - \frac{2 \alpha \gamma}{r^2} \delta_{ij}\right) \exp\left( -\frac{\alpha  R^2}{r^2}\right).
\end{equation}
In view of \eqref{pucci} above we thus find
\begin{equation}
\Mm(D^2 h) = \frac{2\alpha \gamma}{r^2} \underset{A\in [[\la,\Lambda]]}{\inf}\ \left(\frac{2 \alpha}{r^2} \sum_{i,j=1}^n a_{ij} (x_i - \xi_{1}(Q)_{i}) (x_{j} - \xi_{1}(Q)_{j}) - \Sigma  \ a_{ii}\right) \exp\left( - \frac{\alpha R^2}{r^2}\right).
\end{equation}
Therefore,  from \eqref{min}, and by choosing $\alpha$ large enough (depending also on the $C^{1,1}$ character of $\Om$), we have in $S_r(Q,s)$
\begin{align}\label{Mnlower}
\Mm (h) & \geq  \frac{2\alpha \gamma}{r^2} \left(\frac{2 \alpha \lambda}{r^2} |x- \xi_1(Q)|^2 - n\Lambda\right) \exp\left( - \frac{\alpha R^2}{r^2}\right)
\\
& \geq \frac{2\alpha \gamma} {r^2}   \exp\left(- \frac{\alpha R^2}{r^2}\right).
\notag
\end{align}
Hence, from \eqref{lpm} and \eqref{Mnlower} we find in $S_r(Q,s)$
\begin{equation}
\Lm (h) \geq \frac{2\alpha \gamma} {r^2} \left(1 - a  |x - \xi_1(Q)|\right)  \exp\left(- \frac{\alpha R^2}{r^2}\right).
\end{equation}
Now, if $(x,t)\in S_r(Q,s)$ we have  $| x - \xi_1(Q)| \leq r $,  and therefore if we choose $0<r\le \frac{1}{4a}$ we obtain
\begin{equation}
\Lm (h) \geq \frac{3\alpha \gamma} {2 r^2} \exp\left( -\frac{\alpha R^2}{r^2}\right).
\end{equation}
On the other hand, we have  
\[
h_{t} = - \frac{\alpha \gamma} {r^2} \text{sgn}(t-s) \exp\left(- \frac{\alpha R^2}{r^2}\right).
\]
Therefore,  
\begin{equation}\label{h}
\Lm(h) - h_{t} \geq \frac{ \alpha \gamma}{r^2} \exp\left(- \frac{\alpha R^2}{r^2}\right)(3/2 - 1) \geq 0,
\end{equation}
which shows that $h$ is a subsolution to $\Lm - \partial_t$ in the set $(t \neq s)$, i.e., almost everywhere in $S_r(Q,s)$. Hence, $h$ is a strong subsolution in $S_r(Q,s)$. By Proposition 2.11 in \cite{CKS}, we conclude that $h$ is also a viscosity subsolution of $\Lm - \partial_t$ in $S_r(Q,s)$. Since $h (x,t) \leq \frac{u(x,t)}{u(\underline{A}_{r}(Q_0, s_0))}$ on $\partial S_r$,  and by the assumption that $u\in \mathcal S(\la,\Lambda,a)$, $u$ is a supersolution of the same operator, by the comparison principle Corollary \ref{comp} applied to $\Lm - \partial_{t}$, we obtain that
\begin{equation}\label{est2}
h(x,t) \leq \frac{u(x,t)}{u(\underline{A}_{r}(Q_0, s_0))}\ \ in\ \ S_{r}(Q,s).
\end{equation}
At this point, by the argument on p. 290 in \cite{G}, which uses also \eqref{lowernu} above, we conclude that there exists $C_{1} > 0$, depending on $n, \lambda, \Lambda$ and the $C^{1,1}$ character of $\partial \Om$, such that
\begin{equation}\label{est1}
h(x,s) \geq C_{1} \frac{|x-Q|}{r}
\end{equation}
when $x= Q + \delta r \nu_{Q_0}, \delta \in[0,1/32]$, see (3.20) in \cite{G}.
From \eqref{est2}, \eqref{est1}, it follows that for such $x$
\begin{equation}\label{e4}
C_{1} \frac{|x-Q|}{r}u(\underline{A}_{r}(Q_0, s_0)) \leq u(x,s).
\end{equation} 
The estimate \eqref{e4} thus establishes the linear growth from below near the lateral boundary $S_T$ for $u$. 

Our next objective is proving a similar estimate, but from above.
For this, we consider the set $\Psi_{r}(Q_0, s_0) \cap P_{r/8}(\xi_2(Q),s)$, and define 
\[
R_{1}(x,t)^2= |x-\xi_2(Q)|^2 + |t-s|,
\]
where $\xi_2 (Q)$ is defined in \eqref{xi12} above.
If $(x,t) \in \Psi_{r}(Q_0, s_0) \cap P_{r/8}(\xi_2(Q), s)$, we notice that for some constant $C_5>0$ we have
\begin{equation}\label{e5}
\frac{|x-\xi_2(Q)|^2}{R_{1}(x,t)^{2}} \geq C_5.
\end{equation}
For $k>0$ to be subsequently chosen large enough, depending on $n, \lambda, \Lambda, C_{5}$, we define
\[
f(x,t) = 1 - \gamma\left(\frac{r^2}{R_{1}(x,t)^2}\right)^{k},
\]
with $\gamma= \left(\frac{1}{256}\right)^{k}$. Since $R_1(Q,s)^2 = \frac{r^2}{256}$, this choice of $\gamma$ ensures that  $f(Q,s) = 0$. In the next formulas, to abbreviate the notation we write $R_1$ instead of $R_1(x,t)$. A computations gives 
\[
Df =  2 \gamma k \left(\frac{r^2}{R_{1}^2}\right)^{k} \frac{x - \xi_2 (Q)}{R_1 ^2},
\]
\[
D_{ij}f=  - 4 \gamma k \left(\frac{r^2}{R_{1}^2}\right)^{k} \frac{k+1}{R_{1}^4} (x - \xi_{2}(Q))_{i} (x - \xi_{2}(Q))_{j} + 2 \gamma k \left(\frac{r^2}{R_{1}^2}\right)^{k} \frac{\delta_{ij}}{R_1^2}.
\]
Therefore, \eqref{pucci} above gives
\begin{equation}\label{e7}
\Mp (f)=  \frac{2 \gamma k}{R_1^2} \left(\frac{r^2}{R_{1}^2}\right)^{k}  \underset{A\in [[\la,\Lambda]]}{\sup} \left\{-\frac{2(k+1)}{R_{1}^2} a_{ij} (x - \xi_{2}(Q))_{i} (x - \xi_{2}(Q))_{j} + a_{ii}\right\}
\end{equation}
From \eqref{e7} we easily obtain
\[
\Mp (f) \leq \frac{2 \gamma k}{R_1^2} \left(\frac{r^2}{R_{1}^2}\right)^{k} \left\{-\frac{2(k+1)}{R_1^2} \lambda |x- \xi_2(Q)|^2 + n \Lambda\right\}.
\]
Because of \eqref{e5}, we can now choose $k$ large enough so that, for $(x,t) \in \Psi_{r}(Q_0, s_0) \cap P_{r/8}(\xi_2(Q), s)$, we have
\[
\Mp (f)  \leq -  \frac{ 2\gamma k}{R_1^2} \left(\frac{r^2}{R_{1}^2}\right)^{k}.
\]
This estimate and \eqref{lpm} above give 
\begin{equation}\label{e8}
\Lp(f) \leq  -  \frac{2 \gamma k}{R_1^2} \left(\frac{r^2}{R_{1}^2}\right)^{k} \left(1- a | x - \xi_2 (Q)|\right).  
\end{equation}
Observe now that, for $(x,t) \in \Psi_{r}(Q_0, s_0) \cap P_{r/8}(\xi_2(Q), s)$, we have $|x - \xi_2 (Q)| \leq  R_1 < r$. Hence, if $0<r<\frac{1}{4a}$, we obtain from \eqref{e8} 
\[
\Lp(f) \leq  -  \frac 32 \frac{\gamma k}{R_1^2} \left(\frac{r^2}{R_{1}^2}\right)^{k}, \ \ \ \ \ \text{in}\ \Psi_{r}(Q_0, s_0) \cap P_{r/8}(\xi_2(Q), s).
\]
Since 
\[
f_{t} = - \frac{\gamma k}{R_1^2} \left(\frac{r^2}{R_{1}^2}\right)^{k} \text{sgn}(t-s),
\]
we conclude that $\Lp(f) - f_{t} < 0$ when $(t\neq s)$ in $\Psi_{r}(Q_0, s_0) \cap P_{r/8}(\xi_2(Q), s)$. Thus, $f$ is a  strong supersolution, and hence a viscosity supersolution of $\Lp(f) - f_{t}$, by  Proposition 2.11 in \cite{CKS}. Moreover, on $\Psi_{r}(Q_0,s_0) \cap \partial P_{r/8}(\xi_2(Q), s)$, we have
\begin{equation}\label{e9}
f(x,t)= 1 - 2^{-2k},
\end{equation}
while on $\Delta_{r}(Q_0,s_0) \cap P_{r/8}(\xi_2(Q), s)$, we have $f(x,t) \geq 0 $ (as $(Q,s)$ is the closest point where it is zero). The crucial step is to now invoke the Carleson estimate, Theorem \ref{Carleson} above, which gives 
\begin{equation}\label{e10}
\frac{u(x,t)}{u (\overline{A}_r(Q_0,s_0)} \leq C,\ \ \ \ \ (x,t)\in \Psi_{r}(Q_0, s_0) \cap P_{r/8}(\xi_2(Q), s).
\end{equation}
Thus, if $C' = C(1 - 2^{-2k})^{-1}$, then $\frac{u} {u(\overline{A}_r(Q_0,s_0)}\leq  C' f$ on $\partial \left(\Psi_{r}(Q_0,s_0)\cap P_{r/8}(\xi_2(Q),s)\right)$. Since from what we have proved above the function $C'f $ is a viscosity supersolution of $\Lp - \partial_t$ in $\Psi_{r}(Q_0, s_0) \cap P_{r/8}(\xi_2(Q), s)$, whereas, by the hypothesis that $u\in \mathcal S(\la,\Lambda,a)$ in $\Psi_{2r}(Q_{0}, s_{0})$, the function $\frac{u} {u (\overline{A}_r(Q_0,s_0)}$ is a subsolution of the same operator in that set, by Corollary \ref{comp} we conclude that
\begin{equation}\label{e11}
\frac{u} {u(\overline{A}_r(Q_0,s_0)} \leq C' f\ \ \ \ \ \ \ \text{in}\ \Psi_{r}(Q_0,s_0)\cap P_{r/8}(\xi_2(Q),s).
\end{equation}
The estimate \eqref{e11}, combined with the equation (3.33) in \cite{G}, shows that 
\begin{equation}\label{e12}
u(x,s) \le C_{2} \frac{|x-Q|}{r}u(\overline{A}_{r}(Q_0, s_0)),
\end{equation} 
when $x= Q + \delta r \nu_{Q_0}, \delta \in[0,1/32]$. Combining \eqref{e4} and \eqref{e12} we have proved that
\begin{equation}\label{e13}
C_{1} \frac{|x-Q|}{r}u(\underline{A}_{r}(Q_0, s_0)) \le u(x,s) \le C_{2} \frac{|x-Q|}{r}u(\overline{A}_{r}(Q_0, s_0)),
\end{equation} 
when $x= Q + \delta r \nu_{Q_0}, \delta \in[0,1/32]$. In other words \eqref{e13} holds
all over a slice of $\Psi_r(Q_0,s_0)$ which is attached to the lateral boundary $S_T$, and goes inside $C_T$ for a depth
proportional to $r$. The important information \eqref{e13} tells us that, in a $C^{1,1}$ cylinder, any nonnegative function in the class $\mathcal S(\la,\Lambda,a)$ in $\Psi_{2r}(Q_0,t_0)$, which vanishes continuously on $\Delta_{2r}(Q_{0}, s_{0})$, \emph{must vanish at precisely a linear rate}.
 
We next apply \eqref{e13} to the function $v$, obtaining
\begin{equation}\label{e14}
C_{2} \frac{r}{|x-Q|}\frac{1}{v(\overline{A}_{r}(Q_0, s_0))} \le \frac{1}{v(x,s)} \le C_{1} \frac{r}{|x-Q|} \frac{1}{v(\underline{A}_{r}(Q_0, s_0))},
\end{equation} 
when $x= Q + \delta r \nu_{Q_0}, \delta \in[0,1/32]$. The estimates \eqref{e13}, \eqref{e14} immediately imply the desired conclusion \eqref{bc} for points $(x,s)\in \Psi_r(Q_0,t_0)$ such that  $x= Q + \delta r \nu_{Q_0}, \delta \in[0,1/32]$. This allows us to now move away from the lateral boundary $S_T$ by a distance proportional to $r$. Once this is accomplished we can apply the interior Harnack inequality in Theorem \ref{Harnack} above, and thus obtain \eqref{bc} in the remaining portion of the set $\Psi_r(Q_0,t_0)$.
Hence, we obtain \eqref{bc} in $\Psi_r(Q_0,t_0)$ when $r$ additionally satisfies the assumption  $r < \frac{1}{4a}$. Consequently, by an application of the interior Harnack inequality in Theorem \ref{Harnack} to both $u$ and $v$, we conclude that \eqref{bc} continues to be valid for all $r$ small enough depending only on $C^{1,1}$ character of $\Om$, $s_0, T-s_0$ and for a possibly larger universal constant  $C>0$.

\end{proof}

As an application of Theorems \ref{boundary max} and \ref{Boundary comparison theorem}, we obtain the following global comparison theorem for nonnegative functions in the class $\mathcal S(\la,\Lambda,a)$ which vanish on the lateral boundary $S_T$. This result is a consequence of  Theorem \ref{boundary max} and Theorem \ref{Boundary comparison theorem}. Since its proof is identical to that of Theorem 3.2 in \cite{G}, we omit it altogether.

\begin{thrm}[Global comparison theorem]\label{global1}
 Let $C_T = \Om\times (0,T)$ be a $C^{1,1}$ cylinder, and $u, v \ge 0 $  be  functions in $\mathcal S(\la,\Lambda,a)$ in $C_T$, such that $u,v= 0 $ continuously on $S_T$. Fix a point $X_0 \in \Om$, and for a universal $\delta \in (0, \sqrt T/2)$  suitably small (depending on the $C^{1,1}$ character of $\Om$), set $F = \Om \times (2 \delta^2, T- \delta^2)$. There is a universal constant $C>0$, depending also on $\delta$, diam $\Om$, and $T$, such that for all $(x,t)\in F$ one has
\begin{equation}\label{global}
\frac{1}{C} \frac{u(X_0, T)}{v(X_0, T)} \leq \frac{u(x,t)}{v(x,t)} \leq C  \frac{u(X_0, T)}{v(X_0, T)}.
\end{equation}
\end{thrm}

Before stating the next result we recall that functions in the class $\mathcal S(\la,\Lambda,a)$ are invariant under translation in time, see Remark \ref{R:pucci} above.
Consequently, by arguing as in \cite{G}, Theorem 3.3, we have the following backward Harnack inequality for functions vanishing on the lateral boundary $S_T$.

\begin{thrm}[Backward Harnack inequality for $C^{1,1}$ cylinders]\label{backward harnack}
Let $C_T$, $u$, $X_0\in \Om$, $\delta$, $F$ and $C$ be as in Theorem \ref{global1}. Then, for $r$ small enough, one has for $(x,t) \in F$
\begin{equation}\label{bharnack}
u(x, t+ 4r^2) \leq C u(x,t).
\end{equation}
\end{thrm}

\begin{proof}
By Remark \ref{R:pucci} the function $v= u(x, t+ 4r^2)$ belongs to $\mathcal S(\la,\Lambda,a)$. The desired conclusion then follows from Theorems \eqref{global1} and \eqref{boundary max}.

\end{proof}

\begin{rmrk}\label{R:moregeneral}
 We would like to emphasize that we could have treated slightly more general differential inequalities of the form 
\begin{equation}
\Mm (u) - a |Du| - b|u| - u_t \leq u \leq \Mp (u) + a|Du| + b|u| - u_t.
\end{equation}
One can see that the barrier in Theorem \ref{barrier} can be adapted to this situation as well, and thus the  Carleson estimate in Theorem \ref{Carleson} and  the comparison Theorem \ref{Boundary comparison theorem} continue to be valid in this situation. However, for the backward Harnack inequality Theorem \ref {backward harnack}, we crucially make use of the boundary maximum principle Theorem \ref{boundary max} which in turn makes use of the fact that constants  are solutions to the differential inequalities \eqref{maine}. But this is true in the case when $b=0$. Thus, for Theorem \ref{backward harnack}, we do need the assumption $b=0$.
\end{rmrk}

\begin{rmrk}\label{R:final}
It remains a challenging open question whether, for functions in the class $\mathcal S(\la,\Lambda,a)$, Theorem \ref{Boundary comparison theorem} continues to be true for a cylinder $C_T = \Om\times (0,T)$, with a Lipschitz, or more in general non-tangentially accessible (NTA), base $\Om$. To the best of our knowledge, this result  is not known even in the case of fully nonlinear elliptic equations.
\end{rmrk}

\section{Two applications}\label{S:app}

In this section we present two applications of the results in Sections \ref{S:ce} and \ref{S:C11}. We plan to come back to further applications in a future study.

\subsection{Non-divergence form linear equations}

The spirit of this subsection is mainly demonstrative. What we mean by this is that, although the result that we obtain as an application of the results in this paper is not new, the way we obtain it is, and puts it in the broader perspective of the viscosity fully nonlinear theory. In a cylinder $C_T = \Om \times (0,T)$ consider the time-dependent non-divergence, uniformly parabolic linear equation
\begin{equation}\label{linear}
Pu = Lu - u_t =  \sum_{i,j=1}^n a_{ij}(x,t) D_{ij} u  - u_t = 0,
\end{equation}
where for $(x,t)\in C_T$ one has $a_{ij}(x,t)\in [[\la,\Lambda]]$ and, in addition, the $a_{ij}$ are assumed to be continuous. It is well-known, see  for instance Theorem 1.5 in \cite{G}, that if $u$ is a strong solution to \eqref{linear}, then $u \in W^{2,1}_{p, loc}(C_T)$ for all $p > 1$. Since \eqref{linear} is a special case of \eqref{main}, by Proposition 2.11 in \cite{CKS}, $u$ is a viscosity solution to \eqref{linear}. Consequently, the results in this paper imply the backward Harnack inequality, Theorem \ref{backward harnack} above, for  non-negative solutions to \eqref{linear} in $C^{1,1}$ cylinders. We cannot stress enough the non-trivial and remarkable aspect of this observation: we are dealing with a linear parabolic equation, \eqref{linear}, whose coefficients are time-dependent. This serious obstruction can be circumvented by looking at \eqref{linear} from the  point of view of the fully nonlinear class $\mathcal S(\la,\Lambda,a)$ of the Pucci extremal operators, which are instead time-independent!

Now, having said this, we ought to nonetheless mention that, for the linear equation \eqref{linear}, the backward Harnack inequality in Theorem \ref{backward harnack} is in fact available, even in Lipschitz cylinders. This is so thanks to the work, subsequent to \cite{G}, of Fabes, Safonov and Yuan \cite{FSY}, see also the paper \cite{SY}, both of which involve a great deal of hard analysis. In connection with these works, we mention that, when  $u$ is a viscosity solution of the fully nonlinear equation 
\begin{equation}\label{F2}
F(D^2 u,x,t) - u_t=0,
\end{equation}
where $F$ is concave in $D^2 u$ and $C^{1}$, then, by the Evans-Krylov theorem,  $u$ is  a classical $H^{2 + \alpha}$ solution. This follows, for instance, by the existence of classical solutions with continuous boundary values for such $F$, see Theorem 14.10 in \cite{L}, combined with the uniqueness of solutions. Therefore, $u$ solves a linear equation such as \eqref{linear}, where 
\[
a_{ij}(x,t)= \int_0^1 \frac{\partial F}{ \partial M_{ij}}(s D^2 u(x,t), x,t) ds.  
\]
Hence, from the linear theory, see Theorem 3.7 in \cite{FSY}, in the case when $F$ is concave in $M$, then for viscosity solutions of \eqref{F2} the Carleson estimate and the backward Harnack inequality are true in Lipschitz domains.

\subsection{On a degenerate diffusion equation}\label{SS:p}

In this subsection we show how the results in this paper can be applied to obtain new theorems for the nonlinear degenerate diffusion equation
\begin{equation}\label{p}
 u_t = |Du|^{2-p} \text{div}(|Du|^{p-2} Du),\ \ \ \ \ \ \ \ 1<p<\infty. 
\end{equation}
The equation \eqref{p} has recently been  studied in \cite{BG}, \cite{Do}, \cite{MPR}. Such an equation generalizes the motion of level sets by mean curvature, which is relative to the case $p=1$, and the heat equation, which corresponds to $p=2$. In the interesting paper \cite{MPR}, solutions to the equation \eqref{p} have been characterized by asymptotic mean value properties. These properties are connected with the analysis of tug of war games with noise in which the number of rounds in bounded. The value function for these games approximate a solution to the PDE \eqref{p} when the parameter that controls the size of  possible  steps goes to zero.
Let us introduce the relevant notion of sub-, super-, and solution to \eqref{p}. For further details on the properties of viscosity solutions of \eqref{p}, we refer the reader to the paper \cite{BG}.

\begin{dfn}\label{D:p}
A   function  $u\in C(\overline{\Om} \times [0, T))\cap L^{\infty}(\Om \times [0, T))$ is called a \emph{viscosity subsolution} of \eqref{p}  provided that for any $\vf \in  C^{2}(\Om \times (0, T))$ for which 
\begin{equation}
u - \vf\quad  \text{has a local maximum at}\quad   z_0 \in \Om \times  (0, T],
\end{equation}
then, either
\begin{equation}
\begin{cases}
\vf_t \leq \left(\delta_{ij} + (p-2)\frac{D_i \vf D_j \vf}{|D\vf|^{2}}\right)D_{ij} \vf\quad\ \ \text{at}\quad z_0,
\\
\text{if}\   D\vf(z_0) \not= 0,
\end{cases}
\end{equation}
or
\begin{equation}
\begin{cases}
\underset{|a| = 1}{\inf}\ \left(\vf_t - (\delta_{ij} + (p-2)a_{i}a_{j})D_{ij}\vf\right)  \leq 0\ \ \ \text{at}\ z_0,
\\
\text{if}\   D\vf(z_0) = 0.
\end{cases}
\end{equation}
Analogous definitions for supersolution, or solution.
\end{dfn}

Note  that, from Definition \ref{D:p}, at a maximum point of $u - \vf$, we have that
\begin{equation}\label{p2}
\vf_t \leq a_{ij}(D \vf) D_{ij} \vf,
\end{equation}
where 
\begin{equation}\label{lambdas}
\min\{1,p-1\}\ \mathbb I_n \leq [a_{ij}] \leq \max\{1,p - 1\}\ \mathbb{I}_n.
\end{equation}
The inequality \eqref{p2} is seen to be true by considering the matrix $[a_{ij}]$ in both cases when $D\vf$ vanishes and when it does not, see \cite{BG}. It is thus  easily seen from Definition \ref{D:p} and \eqref{lambdas} that any viscosity solution to \eqref{p} belongs to the class $\mathcal S(\la,\Lambda,a)$, with $\lambda = \min\{1,p-1\}$ and $\Lambda= \max\{1,p-1\}$.  Consequently, and quite notably, all the results in this paper, such as the Carleson estimate in Theorem \ref{Carleson}, the local and global comparison Theorems \ref {Boundary comparison theorem} and \ref{global1}, and the backward Harnack inequality in Theorem \ref{backward harnack}, are valid for nonnegative solutions of the degenerate diffusion equation \eqref{p} as well.  Here, we confine ourselves to state, as an example, the following two results. We note that the degenerate parabolic equation \eqref{p} has a non-variational structure and therefore the viscosity framework of this paper is the natural one for studying it.

\begin{thrm}[Local comparison theorem]\label{Boundary comparison theorem p}
Let $C_T = \Om\times (0,T)$ be a $C^{1,1}$ cylinder and let $(Q_{0}, s_{0}) \in S_{T}$. For any $1<p<\infty$, there exists $r_0>0$, depending on $n, p$, the $C^{1,1}$ nature of $\Om$ and on $s_0, T- s_0$, and a  $C>0$, depending on $n, p$ and the $C^{1,1}$ character of $\Om$, such that if $u, v \ge 0$ are two viscosity solutions of \eqref{p} in $\Psi_{2r}(Q_{0}, s_{0})$, and $u=v=0 $ continuously on $\Delta_{2r}(Q_{0}, s_{0})$, then for all $(x,t) \in \Psi_{r}(Q_{0}, s_{0})$, one has
\begin{equation*}
\frac{1}{C} \frac{ u(\underline{A}_{r}(Q_{0}, s_{0}))}{v(\overline{A}_{r}(Q_{0}, s_{0}))}\leq \frac{u(x,t)}{v(x,t)}\leq C \frac{ u(\overline{A}_{r}(Q_{0}, s_{0}))}{v(\underline{A}_{r}(Q_{0}, s_{0}))}.
\end{equation*}
\end{thrm}

\begin{thrm}[Backward Harnack inequality]\label{backward harnack p}
Let $C_T$,  $X_0\in \Om$, $\delta$, $F$ and be as in Theorem \ref{global1} above. Then, given a viscosity solution $u\ge 0$ of \eqref{p} in $C_T$, vanishing continuously on $S_T$, there exist $r_0>0$, depending on $n, p, \delta$, and and the $C^{1,1}$ character of $\Om$, and a constant $C>0$, depending on $n, p$ and the $C^{1,1}$ character of $\Om$, such that for $0<r\le r_0$, one has for $(x,t) \in F$
\[
u(x, t+ 4r^2) \leq C u(x,t).
\]
\end{thrm}

We close by remarking that, in Definition \ref{D:p}, similarly to what was done in \cite{Do}, \cite{BG}, \cite{MPR}, parabolic Euclidean neighborhoods are considered for local extrema. However, by arguing as in Lemma 1.4 in \cite{CKS} (this lemma shows that considering parabolic neighborhoods is equivalent to taking Euclidean neighborhoods), it suffices to consider parabolic neighborhoods at $z_0$. This is what is normally done in the literature, see for instance \cite{W}.

\end{document}